\pdfoutput=1
\documentclass[a4paper,oneside]{amsart}
\usepackage{amsmath,amssymb,amsthm}
\usepackage[english]{babel}
\usepackage[T1]{fontenc}
\usepackage{newtxtext,newtxmath}
\usepackage[utf8]{inputenc}
\usepackage[babel=true]{microtype}
\usepackage[autostyle=true]{csquotes}
\usepackage{thmtools,thm-restate}
\usepackage[pdfusetitle]{hyperref}
\hypersetup{
  colorlinks=true,
  allcolors=blue,
}
\usepackage[citestyle=numeric,bibstyle=numeric,backend=biber,url=false,doi=true,isbn=false,giveninits=true]{biblatex}

\AtEveryBibitem{\clearfield{month}}
\AtEveryBibitem{\clearfield{day}}
\AtEveryBibitem{\clearfield{language}}
\bibliography{literature.bib}

 \usepackage{tikz-cd}
\usepackage[shortlabels]{enumitem}
\newlist{myenumi}{enumerate}{1}
\setlist[myenumi,1]{label=\upshape(\roman*)}
\newlist{myenuma}{enumerate}{1}
\setlist[myenuma,1]{label=\upshape(\alph*)}
\declaretheorem[name=Theorem, numberwithin=section]{thm}
\declaretheorem[name=Theorem, numbered=no]{thm*}
\declaretheorem[name=Lemma,numberlike=thm]{lem}
\declaretheorem[name=Corollary,numberlike=thm]{cor}
\declaretheorem[name=Proposition,numberlike=thm]{prop}

\declaretheorem[name=Example, numberlike=thm, style=remark]{ex}
\declaretheorem[name=Remark, numberlike=thm, style=remark]{rem}
\numberwithin{equation}{section}
\allowdisplaybreaks[1]
\usepackage[noabbrev]{cleveref} 
\crefname{figure}{Figure}{Figures}
\crefname{table}{Table}{Tables}
\crefname{thm}{Theorem}{Theorems}
\crefname{lem}{Lemma}{Lemmas}
\crefname{defi}{Definition}{Definitions}
\crefname{cor}{Corollary}{Corollaries}
\crefname{prop}{Proposition}{Propositions}
\crefname{ex}{Example}{Examples}
\crefname{rem}{Remark}{Remarks}
\crefname{section}{Section}{Sections}
\crefname{chapter}{Chapter}{Chapters}
\crefname{appendix}{Appendix}{Appendices}
\crefdefaultlabelformat{\textup{#2#1#3}}
\creflabelformat{enumi}{\textup{(#2#1#3)}}

\usepackage{xparse}

\NewDocumentCommand \betaPSC {} {
  \beta^{\mathrm{(psc)}}
}

\NewDocumentCommand \betatC {} {
  \beta^{(\mathrm{t})}
}

\newcommand{\betaaC}{\beta^{(\mathrm{a})}}
\newcommand{\betat}{\beta^{(\mathrm{t})}}

\newcommand{\parensup}[1]{\textup{(}#1\textup{)}}

\newcommand{\Riem}{\mathcal{R}}

\newcommand{\fin}{\mathrm{fin}}
\newcommand{\psc}{\mathrm{psc}}
\newcommand{\Fin}{\mathrm{F}}
\newcommand{\tr}{\operatorname{tr}}

\NewDocumentCommand \StolzRel {m m} {
   \mathrm{R}_{#1}^{\mathrm{spin}}
   \left(
      #2
   \right)
}

\NewDocumentCommand \StolzPos {m m} {
   \mathrm{P}_{#1}^{\mathrm{spin}}
   \left(
      #2
   \right)
}

\NewDocumentCommand \SpinBordism {m m} {
   \Omega_{#1}^{\mathrm{spin}}
   \left(
      #2
   \right)
}

\newcommand{\KOtoKU}{\operatorname{\mathbf{c}}}
\newcommand{\cxfy}{\operatorname{c}}
\newcommand{\rlfy}{\operatorname{r}}

\newcommand{\alphaAPS}{\alpha}

\newcommand{\numbers}[1]{\mathbb{#1}}
\newcommand{\C}{\numbers{C}}
\newcommand{\Z}{\numbers{Z}}
\newcommand{\N}{\numbers{N}}
\newcommand{\Q}{\numbers{Q}}
\newcommand{\R}{\numbers{R}}

\newcommand{\iso}{\cong}

\newcommand{\iu}{\mathrm{i}}
\newcommand{\eu}{\mathrm{e}}
\newcommand{\B}{\mathrm{B}}
\newcommand{\E}{\mathrm{E}}

\newcommand{\Eub}{\underline{\mathrm{E}}}

\newcommand{\sphere}{\mathrm{S}}

\newcommand{\ch}{\operatorname{ch}}
\newcommand{\ph}{\operatorname{ph}}
\newcommand{\chub}{\underline{\operatorname{ch}}}
\newcommand{\phub}{\underline{\ph}}

\newcommand{\tens}{\otimes}

\newcommand{\HZ}{\mathrm{H}}
\newcommand{\KO}{\mathrm{KO}}
\newcommand{\KU}{\mathrm{K}}

\newcommand{\RU}{\mathrm{R}}
\newcommand{\RO}{\mathrm{RO}}

\newcommand{\pt}{\mathrm{pt}}

\newcommand{\Cstar}{\mathrm{C}^*}

\newcommand{\SG}{\mathrm{S}}
\newcommand{\CstarRed}{\Cstar_{\mathrm{r}}}

\title{Positive scalar curvature and low-degree group homology}
\author{No\'{e} B\'{a}rcenas }
\email{barcenas@matmor.unam.mx}
\address{Centro de Ciencias Matem\'aticas. UNAM \\Ap.Postal 61-3 Xangari. Morelia, Michoac\'an MEXICO 58089}

\author{Rudolf Zeidler}
\email{math@rzeidler.eu}
\address{Mathematisches  Institut, Westfälische Wilhelms-Universität Münster, Einsteinstr.\ 62, 48149 Münster, Germany}
\begin{document}
   \begin{abstract}
     Let $\Gamma$ be a discrete group.
     Assuming rational injectivity of the Baum--Connes assembly map, we provide new lower bounds on the rank of the positive scalar curvature bordism group and the relative group in Stolz' positive scalar curvature sequence for $\mathrm{B} \Gamma$.
     The lower bounds are formulated in terms of the part of degree up to $2$ in the group homology of $\Gamma$ with coefficients in the $\mathbb{C}\Gamma$-module generated by finite order elements.
     Our results use and extend work of Botvinnik and Gilkey which treated the case of finite groups.
     Further crucial ingredients are a real counterpart to the delocalized equivariant Chern character and Matthey's work on explicitly inverting this Chern character in low homological degrees.
   \end{abstract}
 \maketitle
\section{Introduction}

There exists a natural comparison mapping between the positive scalar curvature (psc) sequence of Stolz \parensup{top row} to the analytic surgery sequence of Higson and Roe \parensup{bottom row}:

\begin{equation}
   \begin{tikzcd}[column sep=small]
      \SpinBordism{n}{\B \Gamma} \rar \dar{\beta} &
      \StolzRel{n}{\B \Gamma} \rar{\partial} \dar{\alpha} &
      \StolzPos{n-1}{\B \Gamma} \rar \dar{\rho} &
      \SpinBordism{n-1}{\B \Gamma} \rar \dar{\beta} &
      \StolzRel{n-1}{\B \Gamma} \dar{\alpha}
      \\
      \KO_{n}(\B \Gamma) \rar{\nu} &
      \KO_{n}(\CstarRed \Gamma) \rar{\partial}
        & \SG^\R_{n-1}(\Gamma) \rar
        & \KO_{n-1}(\B \Gamma) \rar{\nu}
        & \KO_{n-1}(\CstarRed \Gamma)
   \end{tikzcd}\label{eq:mappingPSCtoAnalysis}
\end{equation}

This diagram was first established by \Citeauthor{PS14Rho}~\cite[Theorem~1.28]{PS14Rho} for complex K-theory and $n$ even.
It was extended by \citeauthor{XY14Positive}~\cite[Theorem~B]{XY14Positive} and by the second-named author~\cite[Theorem~3.1.13]{Z16PhD} to cover all dimensions and the real case.
See also \citeauthor{Zen17MappingSurgery}~\cite[Remark 6.2]{Zen17MappingSurgery}.

We briefly explain the constituents in the diagram above.
Start with Stolz' psc sequence.
The group $\SpinBordism{n}{\B \Gamma}$ is the singular \emph{spin bordism group} of the classifying space $\B \Gamma$. That is, it consists of bordism classes of pairs $(M, \phi)$, where $M$ is a closed spin manifold of dimension $n$ and $\phi \colon M \to \B \Gamma$ a continuous map.
The \emph{psc spin bordism group} $\StolzPos{n}{\B \Gamma}$ consists of bordism classes of $(M, \phi, g)$, where $(M, \phi)$ is as before and $g \in \Riem^+(M)$ is a metric of psc.
Here we require bordisms to have metrics of positive scalar curvature with product structure near the boundary.
\emph{Stolz' relative group} $\StolzRel{n+1}{\B \Gamma}$ consists of bordisms classes of triples $(W,\phi, g)$, where $W$ is a compact spin manifold of dimension $(n+1)$, $\phi \colon W \to \B \Gamma$ a continuous map, and $g \in \Riem^+(\partial W)$ a metric of psc on the boundary.
The horizontal arrows in Stolz' sequence are the evident forgetful maps.

The (real version of the) \emph{analytic surgery sequence} of Higson and Roe consists of the \emph{real $\KU$-homology} of $\B \Gamma$, the topological \emph{$\KU$-theory of the reduced group $\Cstar$-algebra} of $\Gamma$, and the \emph{analytic structure group} of $\Gamma$.
We denote the latter by $\SG^\R_\ast(\Gamma)$.
It is defined in such a way that it fits into a long extact sequence together with the real \emph{Novikov assembly map} $\nu \colon \KO_\ast(\B \Gamma) \to \KO_\ast(\CstarRed \Gamma)$.
We will also use their complex counterparts which we denote by $\KU_\ast(\B \Gamma)$, $\KU_\ast(\CstarRed \Gamma)$, and $\SG^\C_\ast(\Gamma)$.

The groups $\StolzPos{n-1}{\B \Gamma}$ and $\StolzRel{n}{\B \Gamma}$ classify psc metrics up to bordism and concordance, respectively, on spin manifolds with fundamental group $\Gamma$.
For the latter see~\cite[Theorem~5.1]{RS01PscSurgery}.
Alas, at present there are no tools known that allow a computation of these groups (not even in simple special cases).
However, the comparison~\labelcref{eq:mappingPSCtoAnalysis} allows us  to obtain lower bounds on these groups using the index-theoretic information contained in the sequence of Higson and Roe.
To that end, it is important to know something about the size of the image of the \emph{relative index map} $\alpha \colon \StolzRel{n}{\B \Gamma} \to \KO_{n}(\CstarRed \Gamma)$ and the \emph{higher rho-invariant} $\rho \colon \StolzPos{n-1}{\B \Gamma} \to \SG_{n-1}^\R(\Gamma)$.

The first case where something can be said is the class of finite groups.
Indeed, let $H$ be a finite group.
Let $\RU(H)$ denote its complex representation ring.
Let $\RU_0^q(H)$ be the subgroup generated by those representations $\rho$ of virtual dimension $0$ such that its character $\chi_\rho$ satisfies $\chi_\rho(\gamma^{-1}) = (-1)^q \chi_\rho(\gamma)$ for all $\gamma \in H$.
\Citeauthor{BG95Eta}~\cite[Theorem~2.1]{BG95Eta} showed that the rank of the positive scalar curvature bordism group $\StolzPos{2q + 4k -1}{\B H}$ is bounded below by the rank of $\RU^q_0(H)$, where $k \geq 1$, $q \in \{0,1\}$ with $4k + 2q \geq 6$.
They used relative eta-invariants.
These are numerical invariants that are related to the higher $\rho$-invariant via certain trace maps, see for instance~\cite{HR10Eta}.
In fact, \citeauthor{BG95Eta}'s result implies that both $\rho \colon \StolzPos{n-1}{\B H} \to \SG_{n-1}^\R(H)$ and $\alpha \colon \StolzRel{n}{\B H} \to \KO_{n}(\CstarRed H)$ are \emph{rationally surjective} for $n \geq 6$ (we explain this in \cref{sec:finiteGroups}).
Moreover, recently \citeauthor{WY15FinitePart}~\cite{WY15FinitePart} and \citeauthor{XY17Moduli}~\cite{XY17Moduli} gave lower bounds for a large class of infinite groups based on the number of torsion elements with pairwise different orders.
We also refer to~\citeauthor{PS07Torsion}~\cite{PS07Torsion} for lower bounds on the positive scalar bordism group based on the $L^2$-rho-invariant.

The results mentioned above only yield information for $n$ even.
Using product formulas one can obtain further ad-hoc examples of non-trivial relative indices and $\rho$-invariants by taking certain products, see~\cite[Corollary~6.10]{Z16Positive} and \cite[Corollary~5.24]{Zen17MappingSurgery}.
For instance, taking products with the circle allows to also produce examples for $n$ odd.

In the main results of this paper, we give new systematic lower bounds for all $n \geq 7$ on the image of the relative index and the rho-invariant based on the part of degree up to $2$ of a certain group homology.
The result of \citeauthor{BG95Eta} yields the $0$-dimensional part.
Then the idea ist that degrees $1$ and $2$ can be obtained from this by taking products with circles and oriented surfaces.
We use the \emph{Baum--Connes assembly map} $\mu \colon \KU_\ast^\Gamma(\Eub \Gamma) \to \KU_\ast(\CstarRed \Gamma)$, the \emph{delocalized Chern character} of \citeauthor{BC88Chern}~\cite{BC88Chern}, and --- most centrally --- its explicit partial inverse in degrees up to $2$ due to \citeauthor{matthey:delocChern}~\cite{matthey:delocChern}.

To state our results, we start with some preparations.
Let $\Gamma$ be a discrete group and denote by $\Gamma_\fin$ the set of elements of finite order of $\Gamma$.
Let $\Fin \Gamma$ be the set of all finitely supported functions $\Gamma_\fin \to \C$.
  Letting $\Gamma$ act by conjugation on $\Gamma_\fin$ turns $\Fin \Gamma$ into a $\C \Gamma$-module.
  The delocalized equivariant Chern character yields an isomorphism
  \begin{equation}
    \chub_\Gamma \colon \KU_p^\Gamma(\Eub \Gamma) \tens \C \xrightarrow{\cong} \bigoplus_{k \in \Z} \HZ_{p + 2k}(\Gamma; \Fin \Gamma)\label{eq:delocChern}
  \end{equation}
    It was first introduced by \citeauthor{BC88Chern}~\cite{BC88Chern} but we will instead work with the \enquote{handicrafted Chern character} of \citeauthor{matthey:delocChern}~\cite[Theorem~1.4]{matthey:delocChern}.
    \Citeauthor{matthey:delocChern} also constructed maps $\betatC_p \colon \HZ_p(\Gamma; \Fin \Gamma) \to \KU_p^\Gamma(\Eub \Gamma) \tens \C$ for $p \in \{0,1,2\}$ which are right-inverse to the delocalized Chern character.
    Moreover, he defined explicit maps $\betaaC_p \colon \HZ_p(\Gamma; \Fin \Gamma) \to \KU_p(\CstarRed \Gamma)$ which satisfy $\betaaC_p \circ (\mu \tens \C) = \betatC_p$ for $p \in \{0,1,2\}$.
    They thereby describe the Baum--Connes assembly map explicitly in low homological degrees.

    To use these maps for our purposes, we need to adapt the above to real K-homology.
    To that end, for $q \in \{0,1\}$, let $\Fin^q \Gamma = \{f \in \Fin \Gamma \mid f(\gamma) = (-1)^q f(\gamma^{-1})\ \forall \gamma \in \Gamma_\fin\}$.
    Then $\Fin \Gamma = \Fin^0 \Gamma  \oplus \Fin^1 \Gamma$ as $\C \Gamma$-modules.
    We can now state our main result and its corollaries.
\begin{thm}\label{thm:betaPSC}
  For each $p \in \{0,1,2\}$, $q \in \{0,1\}$ and $k \geq 1$ with $4k + 2q \geq 6$, there exists a linear map
  \begin{equation*}
    \betaPSC_{p,q,k} \colon \HZ_p(\Gamma; \Fin^q \Gamma) \to \StolzRel{p + 2q + 4k}{\B \Gamma} \tens \C,
  \end{equation*}
  such that the following diagram commutes
  \begin{equation*}
    \begin{tikzcd}
      \HZ_p(\Gamma; \Fin^q \Gamma ) \rar[hookrightarrow] \dar["\betaPSC_{p,q,k}"]
        & \HZ_p(\Gamma; \Fin \Gamma) \ar[dd, "\betaaC_p"] \\
      \StolzRel{p + 2q + 4k}{\B \Gamma} \tens \C \dar["\alpha \tens \C"] & \\
      \KO_{p+2q}(\CstarRed \Gamma)  \tens \C \rar["\cxfy \tens \C"]  & \KU_p(\CstarRed \Gamma) \tens \C
    \end{tikzcd}
  \end{equation*}
\end{thm}
Here $\cxfy \colon \KO_\ast(\CstarRed \Gamma) \to \KU_\ast(\CstarRed \Gamma)$ is the complexification map from real to complex K-theory.
We implicitly use that complex K-theory is $2$-periodic and real K-theory is rationally $4$-periodic.

\begin{rem}
  We do not claim that our maps $\betaPSC_{p,q,k}$ are canonical (unlike the original maps of Matthey).
  Indeed, their construction depends on choosing pre-images under the surjective map $\alpha \tens \Q \colon \StolzRel{\ast}{\B H} \tens \Q \to \KO_\ast(\CstarRed H) \tens \Q$ for each finite cyclic group $H$.
  However, after fixing these choices it is in principle possible to trace through the construction to obtain explicit formulas for $\betaPSC_{p,q,k}$ similarly as in Matthey's work.
\end{rem}

  In any case, the existence of $\betaPSC_{p,q,k}$ implies lower bounds and surjectivity results:

  \begin{cor}\label{cor:RlowerBound}
    Let $n \geq 7$.
    If the rational Baum--Connes assembly map $\mu \otimes \Q$ is injective, then the rank of $\StolzRel{n}{\B \Gamma}$ is at least the dimension of

    \begin{equation*}
      \begin{cases}
        \HZ_0(\Gamma; \Fin^0 \Gamma) \oplus \HZ_2(\Gamma; \Fin^1 \Gamma) & n \equiv 0 \mod 4,\\
        \HZ_1(\Gamma; \Fin^0 \Gamma) & n \equiv 1 \mod 4,\\
        \HZ_0(\Gamma; \Fin^1 \Gamma) \oplus \HZ_2(\Gamma; \Fin^0 \Gamma) & n \equiv 2 \mod 4,\\
        \HZ_1(\Gamma; \Fin^1 \Gamma) & n \equiv 3 \mod 4.
      \end{cases}
    \end{equation*}

  \end{cor}

  \begin{cor}\label{cor:PlowerBound}
    Let $n \geq 7$.
    If the rational Baum--Connes assembly map $\mu \otimes \Q$ is injective, then the rank of $\StolzPos{n-1}{\B \Gamma}$ is at least the dimension of
    \begin{equation*}
    \begin{cases}
      \HZ_0(\Gamma; \Fin^0_0 \Gamma) \oplus \HZ_2(\Gamma; \Fin^1 \Gamma) & n \equiv 0 \mod 4,\\
      \HZ_1(\Gamma; \Fin^0_0 \Gamma) & n \equiv 1 \mod 4,\\
      \HZ_0(\Gamma; \Fin^1 \Gamma) \oplus \HZ_2(\Gamma; \Fin^0_0 \Gamma) & n \equiv 2 \mod 4,\\
      \HZ_1(\Gamma; \Fin^1 \Gamma) & n \equiv 3 \mod 4,
    \end{cases}
  \end{equation*}
  where $\Fin_0^0 = \{f \in \Fin^0 \mid f(1) = 0 \}$.
  \end{cor}
    In comparison, \citeauthor{BG95Eta}~\cite[Theorem~0.1]{BG95Eta} imply that for a finite group $H$ and $n \geq 6$ even, the rank of $\StolzPos{n-1}{\B H}$ is bounded below by the dimension of $\HZ_0(H; \Fin^0_0 H)$ if $n \equiv 0 \mod 4$ or $\HZ_0(H; \Fin^1 H)$ if $n \equiv 2 \mod 4$.
    
      \begin{ex}
    We describe an explicit example illustrating the non-trivial content of \cref{cor:RlowerBound,cor:PlowerBound}.
    Let $\Sigma_g$ denote the oriented surface of genus $g \geq 1$.
    Let $n$ be a positive integer.
    Consider the group $\Gamma = \pi_1(\Sigma_g) \times \Z /n\Z$.
    Then the Baum--Connes assembly map for $\Gamma$ is an isomorphism\footnote{This follows readily from~\citeauthor{HK01BCaTmenable}~\cite{HK01BCaTmenable} because $\pi_1(\Sigma_g) \times \Z/n\Z$ is a-T-menable. However, we should note that the case of surface groups goes back to the original article of \citeauthor{BC00GeometricKTheory}~\cite{BC00GeometricKTheory}.}.
    So our results are applicable.
    Next we explicitly compute the homology groups that appear in the corollaries for this example.
    Start with the group homology of $\Gamma$ with trivial coefficients $\C$.
    The group homology of $\Z/n\Z$ is torsion in all positive degrees.
    Hence the Künneth theorem implies that $\HZ_\ast(\Gamma; \C) \iso \HZ_\ast(\Sigma_g; \C)$.
    Thus the homology of $\Gamma$ is $\HZ_p(\Gamma; \C) \iso \C$ for $p \in \{0,2\}$, $\HZ_1(\Gamma; \C) \iso \C^{2g}$, and zero in degrees greater than $2$.    
    To proceed, observe that any finite order element of $\Gamma$ is of the form $(1, t^k)$, where $1$ denotes the neutral element of $\pi_1(\Sigma_g)$ and $t$ the generator of $\Z/n\Z$.
    The action by conjugation is trivial on these elements.
    We deduce that $\Fin \Gamma$ is isomorphic to the trivial $\C \Gamma$-module $\C^{n}$.
    By counting dimensions, we see that $\Fin^0 \Gamma \iso \C^{\left\lfloor \frac{n}{2} \right\rfloor + 1}$, $\Fin^0_0 \Gamma \iso \C^{\left\lfloor \frac{n}{2} \right\rfloor}$, $\Fin^1 \Gamma \iso \C^{\left\lceil \frac{n}{2} \right\rceil - 1}$ as trivial $\C \Gamma$-modules.
    Together with the computation of $\HZ_\ast(\Gamma; \C)$, we deduce:
    \begin{gather*}
      \HZ_0(\Gamma; \Fin^0 \Gamma) \iso \C^{\left\lfloor \frac{n}{2} \right\rfloor + 1}, \quad 
        \HZ_0(\Gamma; \Fin^0_0 \Gamma) \iso \C^{\left\lfloor \frac{n}{2} \right\rfloor}, \quad
        \HZ_2(\Gamma; \Fin^1 \Gamma) \iso \C^{\left\lceil \frac{n}{2} \right\rceil - 1}, \\
       \HZ_1(\Gamma; \Fin^0 \Gamma) \iso \C^{2g \left(\left\lfloor \frac{n}{2} \right\rfloor+1 \right)}, \quad \HZ_1(\Gamma; \Fin^0_0 \Gamma) \iso \C^{2g \left\lfloor \frac{n}{2} \right\rfloor},\\
       \HZ_0(\Gamma; \Fin^1 \Gamma) \iso \C^{\left\lceil \frac{n}{2} \right\rceil - 1}, \quad 
        \HZ_2(\Gamma; \Fin^0 \Gamma) \iso \C^{\left\lfloor \frac{n}{2} \right\rfloor + 1}, \quad
        \HZ_2(\Gamma; \Fin^0_0 \Gamma) \iso \C^{\left\lfloor \frac{n}{2} \right\rfloor}, \\
       \HZ_1(\Gamma; \Fin^1 \Gamma) \iso \C^{2g \left(\left\lceil \frac{n}{2} \right\rceil-1 \right)}.
    \end{gather*}
    This shows that for $n \geq 3$ all homology groups which appear in the conclusion of \cref{cor:RlowerBound,cor:PlowerBound} are non-trivial.    
  \end{ex}

  \begin{cor}\label{cor:surj}
    Let $n \geq 7$.
    Let the rational homological dimension of $\Gamma$ be at most $2$.
    Then, if the rational Baum--Connes assembly map $\mu \tens \Q$ is surjective, the rational relative index map
    \begin{equation*}
      \alpha \otimes \Q \colon \StolzRel{n}{\B \Gamma} \otimes \Q \twoheadrightarrow \KO_n(\CstarRed \Gamma) \otimes \Q
    \end{equation*}
    is surjective.

    If $\mu \tens \Q$ is an isomorphism, then the rational higher rho-invariant
      \begin{equation*}
       \rho \otimes \Q \colon \StolzPos{n-1}{\B \Gamma} \otimes \Q \twoheadrightarrow \SG_{n-1}^\R(\Gamma) \otimes \Q
    \end{equation*}
    is also surjective.
  \end{cor}
\section{The delocalized equivariant Pontryagin character}
In this section, we exhibit the \emph{delocalized equivariant Pontryagin character}, which is the real counterpart to the delocalized equivariant Chern character.
It is obtained from the delocalized Chern character by precomposing it with complexification.

Start with some preparations. The rationalized equivariant real K-homology $\KO_\ast^{\bullet} \tens \Q$ is $4$-periodic.
Indeed, it is a module over $\KO_\ast(\pt) \tens \Q \iso \Q[\alpha, \beta, \beta^{-1}] / \langle \alpha^2 - 4 \beta \rangle$ with $\alpha \in \KO_4(\pt)$, $\beta \in \KO_8(\pt)$ and module multiplication with $\alpha / 2$ implements the $4$-periodicity.
We will implicitly use this $4$-periodicity whenever convenient.

The complexification $\cxfy \colon \KO_\ast(\pt) \to \KU_\ast(\pt) \iso \Z[\xi, \xi^{-1}]$ satisies $\cxfy(\alpha) = 2 \xi^2$ and $\cxfy(\beta) = \xi^4$, where $\xi \in \KU_2(\pt)$.

Complex K-homology rationally decomposes into two copies of real K-homoloy:

\begin{prop}\label{prop:KOtoKU}
Complexification yields an isomorphism of proper equivariant homology theories:
\begin{equation}
 \KOtoKU := \cxfy + \xi^{-1} \cxfy \colon  \left( \KO^\bullet_\ast \oplus \KO^\bullet_{\ast+2} \right) \tens \Q \xrightarrow{\iso} \KU^\bullet_\ast \tens \Q.\label{eq:KOtoKU}
\end{equation}
\end{prop}

The decomposition \labelcref{eq:KOtoKU} can be used to decompose the equivariant delocalized Chern character and thereby obtain the \emph{delocalized Pontryagin character}:
\begin{prop} \label{prop:realEquChern}
The equivariant delocalized Chern character composed with complexification yields an isomorphism
\begin{align*}
  \phub_\Gamma := \chub_\Gamma \circ \cxfy \colon \KO_p^\Gamma(\Eub \Gamma) \tens \C \xrightarrow{\cong}
  \bigoplus_{k \in \Z} \HZ_{p+4k}(\Gamma; \Fin^0\Gamma) \oplus \HZ_{p + 2 + 4k}(\Gamma; \Fin^1 \Gamma).
\end{align*}
\end{prop}

Since Matthey's maps are right-inverse to the delocalized Chern character, \cref{prop:realEquChern} immediately implies that they decompose as follows:
\begin{cor} \label{prop:mattheyKO}
Using the identification~\labelcref{eq:KOtoKU}, Matthey's inverse maps~\textup{\cite{matthey:delocChern}}
\begin{equation*}
\betatC_p \colon \HZ_p(\Gamma; \Fin\Gamma) \to \KU_{p}^\Gamma(\Eub{\Gamma}) \tens \C, \qquad p \in \{0,1,2\},
\end{equation*}
restrict to maps
\begin{align*}
\betatC_{p,q} \colon \HZ_p(\Gamma; \Fin^q\Gamma) \to \KO_{p+2q}^\Gamma(\Eub{\Gamma}) \tens \C, \qquad p \in \{0,1,2\}, q \in \{0,1\}.
\end{align*}
\end{cor}
Analogously as in the complex case we write $\betaaC_{p,q} := \mu \tens \C \circ \betat_{p,q}$.
By abuse of notation $\mu \colon \KO_{p+2q}^\Gamma(\Eub \Gamma) \to \KO_{p +2q}(\CstarRed \Gamma)$ denotes the real version of the Baum--Connes assembly map here.

The proofs of \cref{prop:KOtoKU,prop:realEquChern} can essentially be reduced to the case of finite groups.
Thus we need some facts about equivariant K-homology of a finite group and fix some notation.
For more details on the following we refer for example to~\cite[Chapter 2]{BG10Real}:

Let $\RU(H)$ denote the complex representation ring and $\RO(H)$ its real counterpart.
Let $\cxfy \colon \RO(H) \to \RU(H)$ be complexification and $\rlfy \colon \RU(H) \to \RO(H)$ be realification.
Let $\tau \colon \RU(H) \to \RU(H)$ be the map induced by complex conjugation.
Then $\cxfy \circ \rlfy = 1 + \tau$ and $\rlfy \circ \cxfy = 2$.
Let $\RU(H)/(1+\tau)$ be a shorthand for the quotient group $\RU(H)/(1+\tau)(\RU(H))$.
Then the equivariant K-homology of a point satisfies
\begin{align}
  \KO_i^H(\pt) \tens \Q &\iso
    \begin{cases}
      \RO(H) \tens \Q & i \equiv 0 \mod 4,\\
      \RU(H)/ (1 + \tau) \tens \Q & i \equiv 2 \mod 4,\\
      0 & \text{otherwise},
    \end{cases}\label{eq:equKO}\\
    \KU_i^H(\pt) \tens \Q &\iso
      \begin{cases}
        \RU(H) \tens \Q & i \equiv 0 \mod 2,\\
        0 & \text{otherwise}.
      \end{cases}\label{eq:equKU}
\end{align}
Complexification $\cxfy \colon \KO_i^H(\pt) \to \KU_i^H(\pt)$ is given by complexification of representations for $i \equiv 0 \mod 4$, and by the map $1 - \tau \colon \RU(H)/(1+\tau) \to \RU(H)$ for $i \equiv 2 \mod 4$.

We are now ready to prove the propositions of this section.

\begin{proof}[Proof of \cref{prop:KOtoKU}]
  To show that $\KOtoKU$ is an isomorphism of proper equivariant homology theories, it suffices to show that
  \begin{equation}
     \KOtoKU = \cxfy + \xi^{-1} \cxfy \colon  \left( \KO^H_i(\pt) \oplus \KO^H_{i+2}(\pt) \right) \tens \Q \rightarrow \KU^H_i(\pt) \tens \Q
     \label{eq:KOtoKUcoeff}
  \end{equation}
  is an isomorphism for every finite group $H$ and every $i \in \Z$.
  The map  $\left( \KO^\bullet_i \oplus \KO^\bullet_{i+2} \right) \tens \Q \to  \left( \KO^\bullet_{i+2} \oplus \KO^\bullet_{i+4} \right) \tens \Q$, $x \oplus y \mapsto y \oplus (\alpha/2) \cdot x$ defines a $2$-periodicity on the left-hand side which corresponds to Bott periodicity after applying $\KOtoKU = \cxfy + \xi^{-1} \cxfy$.
  Hence it suffices to check \labelcref{eq:KOtoKUcoeff} for $i \in \{0,1\}$.
  For $i = 1$, both sides of \labelcref{eq:KOtoKUcoeff} are zero by \labelcref{eq:equKO,eq:equKU}.
  It remains to check $i=0$.
  In this case $\KOtoKU$ correponds to the map $(\RO(H) \oplus \RU(H)/(1+\tau)) \tens \Q \to \RU(H) \tens \Q$, $x \oplus [y] \mapsto \cxfy(x) + y - \tau(y)$.
  This is an isomorphism because the map $\RU(H) \tens \Q \to (\RO(H) \oplus \RU(H)/(1+\tau)) \tens \Q$, $z \mapsto \frac{1}{2} (\rlfy(z) \oplus [z])$ is an explicit inverse.
\end{proof}
\begin{proof}[Proof of \cref{prop:realEquChern}]
  Start with the (non-equivariant) Pontryagin character isomorphism
  \begin{equation}
    \ph = \ch \circ \cxfy \colon \KO_p(X) \tens \Q \xrightarrow{\iso} \bigoplus_{k \in \Z}\HZ_{p + 4k}(X; \Q). \label{eq:phNonEqu}
  \end{equation}
It is, by definition, the Chern character applied after complexification.
  In particular, the proposition holds if $\Gamma$ is torsion-free.

  Next we deal with the case where $\Gamma = H = \Z/n\Z$ is a finite cyclic group.
  Then the Chern character yields an isomorphism $\chub_H \colon \RU(H) = \KU_0^H(\pt) \to \HZ_{0}(H; \Fin H)$ as all the other homology groups vanish.
  We have $\HZ_{0}(H; \Fin H) = \Fin H = \C^H$ and $\chub_H$ associates to a representation $[\rho] \in \RU(H)$ its character $\chi_\rho$.
  As the character of any real representation is symmetric and the character of an element in the image of $(1 - \tau) \colon \RU(H) \to \RU(H)$ is antisymmetric, we conclude that
  \begin{equation}
    \chub_\Gamma(\cxfy(\KO_{2q}^H(\pt))) = \HZ_{0}(H; \Fin^q H).
    \label{eq:phFinite}
  \end{equation}

  Next we let $G$ be some group and consider $\Gamma := G \times \Z/ n \Z$ and $y \times z \in \KO_{p+2q}^\Gamma(\Eub \Gamma)$ with $y \in \KO_p(\B G) \iso \KO_p^G(\E G)$ and $z \in \KO_{2q}^{\Z / n \Z}(\pt)$.
  There is a natural isomorphism $\KU_p^{G \times \Z/n\Z}(\Eub G) \iso \KU^G_p(\Eub G) \tens \RU(\Z/n\Z) \iso \KU^G_p(\Eub G) \tens \KU_0^{\Z/n\Z}(\pt)$ and a commutative diagram
  \begin{equation*}
     \begin{tikzcd}
        \KU_p(\B G) \tens \KU^{\Z/n\Z}_0(\pt) \tens \C \rar{\times} \dar{\ch \tens \chub_{\Z/n\Z}}
            & \KU_p^{G \times \Z/n\Z}(\Eub G) \dar{\chub_\Gamma} \tens \C \\
        \bigoplus_{k \in \Z}\HZ_{p+2k}(G; \C) \tens \HZ_{0}(\Z/n\Z; \Fin \Z/n\Z)  \rar{\times} & \bigoplus_{k \in \Z}\HZ_{p+2k}(\Gamma; \Fin \Gamma).
     \end{tikzcd}
  \end{equation*}
In view of \labelcref{eq:KOtoKU,eq:phNonEqu,eq:phFinite} this diagram restricts to
  \begin{equation*}
     \begin{tikzcd}
        \KO_p(\B G) \tens \KO^{\Z/n\Z}_{2q}(\pt) \tens \C \rar{\times} \dar{\ph \tens \chub_{\Z/n\Z} \circ \cxfy}
            & \KO_{p+2q}^{G \times \Z/n\Z}(\Eub G) \dar{\chub_\Gamma \circ \cxfy} \tens \C \\
        \bigoplus_{k \in \Z}\HZ_{p+4k}(G; \C) \tens \HZ_{0}(\Z/n\Z; \Fin^q \Z/n\Z)  \rar{\times} & \bigoplus_{k \in \Z}\HZ_{p+4k}(\Gamma; \Fin^q \Gamma).
     \end{tikzcd}
  \end{equation*}
 We conclude
 \begin{equation}
   \chub_\Gamma(\cxfy(y \times z)) \in \bigoplus_{k \in \Z}\HZ_{p + 4k}(\Gamma; \Fin^q \Gamma).\label{eq:phProduct}
   \end{equation}

 Now let $\Gamma$ be general.
 The equivariant $\KU$-homology $\KU_p^\Gamma(\Eub \Gamma)$ is generated by elements of the form $\varphi_\ast(y \times z)$ with $G \subseteq \Gamma$ some subgroup, $y \in \KU_p(\B G)$, $z \in \KU_0^{\Z/n\Z}(\pt)$ and $\varphi \colon G \times \Z/n\Z \to \Gamma$ some group homomorphism.
 This follows from~\cite[Theorem~1.3 and Section~7]{matthey:delocChern}.
 Using \labelcref{eq:KOtoKU}, this implies that $\KO_i^\Gamma(\Eub \Gamma)$ is generated by elements of the form $\varphi_\ast(y \times z)$ with $y \in \KO_p(\B G)$, $z \in \KO_{2q}^{\Z/n\Z}(\pt)$ such that $i = p + 2q$.
 Thus \labelcref{eq:phProduct} implies
 \begin{equation}
   (\chub_\Gamma \circ \cxfy)(\KO_p^\Gamma(\Eub \Gamma)) \subseteq \bigoplus_{k \in \Z} \HZ_{p+4k}(\Gamma; \Fin^0\Gamma) \oplus \HZ_{p + 2 + 4k}(\Gamma; \Fin^1 \Gamma) \label{eq:phGeneral}
   \end{equation}
   Now the proposition follows by combining \labelcref{eq:KOtoKU,eq:phGeneral} and using the fact that $\chub_\Gamma$ is an isomorphism.
\end{proof}

\section{Matthey's maps}\label{sec:mattheysMaps}
In this section, we exhibit the real versions of Matthey's maps from \cref{prop:mattheyKO} more explicitly.
We start with a brief summary of the material in~\cite{matthey:delocChern} that leads to \citeauthor{matthey:delocChern}'s definition of
\begin{equation*}
  \betatC_p \colon \HZ_p(\Gamma; \Fin(\Gamma)) \to \KU_{p}^\Gamma(\Eub{\Gamma}) \tens \C, \qquad p \in \{0,1,2\}.
\end{equation*}
Let $G^{(0)} := 1$, $G^{(1)} := \Z$ and $G^{(2),g} := \Gamma_g := \pi_1(\Sigma_g)$, where $\Sigma_g$ is the oriented surface of genus $g \geq 1$.
To simplify the notation, we let $G^{(2)}$ stand for $G^{(2),g}$ for some $g$.
Moreover, let $G_n^{(p)} := G^{(p)} \times \Z / n \Z$.

There is an isomorphism $\HZ_p(\Gamma; \Fin \Gamma) \iso \bigoplus_{C} \HZ_p(\B Z_C; \C)$.
Here the direct sum runs over all conjugacy classes of finite order elements and $Z_C$ denotes the centralizer of some element in the conjugacy class $C$.
An element $x \in \HZ_p(\Gamma; \Fin \Gamma)$ is called \emph{homogeneous} if it is contained in one of the direct summands.
For each $p \in \{0,1,2\}$ and $n \geq 0$, there is a certain fundamental class $[G_n^{(p)}] \in \HZ_p(G_n^{(p)}; \Fin G_n^{(p)})$.
These have the property that for $p \in \{0,1,2\}$ any homogeneous element $x \in \HZ_p(\Gamma; \Fin \Gamma)$ can be written as $x = \phi_\ast [G_n^{(p)}]$ for some $G_n^{(p)}$ and some group homomorphism $\phi \colon G_n^{(p)} \to \Gamma$.

Moreover, there is a certain $\KU$-homological fundamental class $[G_n^{(p)}]_{\KU} \in \KU_p^{G_n^{(p)}}(\Eub{G_n^{(p)}}) \tens \C$.
Setting $\phi_\ast [G_n^{(p)}] \mapsto \phi_\ast [G_n^{(p)}]_{\KU}$ yields a well-defined map $\HZ_p(\Gamma; \Fin(\Gamma)) \to \KU_{p}^\Gamma(\Eub{\Gamma}) \tens \C$ that is right-inverse to the equivariant Chern character.
This is Matthey's definition of $\betatC_p$.

To describe how this map decomposes in the real case, we need to recall the definition of the $\KU$-homological fundamental class $[G_n^{(p)}]_{\KU}$.
We start with the fundamental class $[G^{(p)}]_\KU \in \KU_p(\Eub{G^{(p)}})$ which are defined as $[G^{(0)}]_\KU := 1 \in \KU_0(\Eub{G^{(0)}}) = \KU_0(\pt)$, $[G^{(1)}]_\KU := [\sphere^1]_\KU \in \KU_1(\sphere^1) \iso \KU_1^\Z(\Eub \Z)$ and $[G^{(2)}]_\KU := [\Sigma_g]_\KU \in \KU_2(\Sigma_g) \iso \KU_2^{\Gamma_g}(\Eub \Gamma_g)$.
That is, $[G^{(p)}]_\KU$ is the $\KU$-homological fundamental class of the point, the circle or an oriented surface of positive genus.
Observe that we may take $\Eub{G^{(p)}_n} = \Eub{G^{(p)}}$ by letting $\Z / n \Z$ act trivially.
Then we set
\begin{equation*}
  [G_n^{(p)}]_\KU := \sum_{l=0}^{n-1} ([G^{(p)}]_\KU \times [\omega_n^l] ) \tens \omega^{-l}_n \in \KU_p^{G_n^{(p)}}(\Eub{G_n^{(p)}}) \tens \C.
\end{equation*}
In this formula $\omega_n := \eu^{2 \pi \iu / n}$ denotes the primitive $n$-th root of unity and $[\omega_n^l] \in \RU(\Z/n\Z) = \KU_0^{\Z/n\Z}(\pt)$ the corresponding representation.

To obtain the real counterparts to this, we first observe how $[\omega_n^l] \in \KU_0^{\Z/n\Z}(\pt)$ decomposes under the isomorphism from \cref{prop:KOtoKU} and \labelcref{eq:equKO,eq:equKU}.
Indeed,
\begin{equation*}
  \KU_0^{\Z/n\Z}(\pt) \tens \Q \iso (\KO_0^{\Z/n\Z}(\pt) \oplus \KO_2^{\Z/n\Z}(\pt)) \tens \Q \iso
  (\RO(\Z/n\Z) \oplus \RU(\Z / n \Z) / (1 + \tau) )\tens \Q.
\end{equation*}
Given $x \in \RU(Z/n\Z) \tens \Q$, we will write $\Re x := \rlfy(x) / 2 \in \RO(\Z/n\Z) \tens \Q$ and $\Im x \in \RU(\Z/n\Z) / (1 + \tau) = \KO_2^{\Z/n\Z}(\pt)$ for the class represented by $x / 2$.
Then we have $x = \cxfy(\Re x + \Im x)$.
We define
\begin{align}
  [G_n^{(p)}]_{\KO}^0 &:= \sum_{l=0}^{n-1} ([G^{(p)}]_\KO \times \Re [\omega_n^l] ) \tens \omega^{-l}_n \in \KO_p^{G_n^{(p)}}(\Eub{G_n^{(p)}}) \tens \C, \nonumber \\
  [G_n^{(p)}]_{\KO}^1 &:= \sum_{l=0}^{n-1} ([G^{(p)}]_\KO \times \Im [\omega_n^l] ) \tens \omega^{-l}_n \in \KO_{p+2}^{G_n^{(p)}}(\Eub{G_n^{(p)}}) \tens \C.
  \label{eq:equFundClass}
\end{align}
Here $[G^{(p)}]_\KO$ denotes the $\KO$-fundamental class of the point, the circle or a surface, respectively.
We find that $[G_n^{(p)}]_\KU = [G_n^{(p)}]_{\KO}^0 \oplus [G_n^{(p)}]_{\KO}^1$ under the isomorphism from \cref{prop:KOtoKU}.
The homological fundamental class also decomposes as $[G_n^{(p)}] = [G_n^{(p)}]^0 \oplus [G_n^{(p)}]^1$ according to $\HZ_p(G_n^{(p)}; \Fin G_n^{(p)}) = \HZ_p(G_n^{(p)}; \Fin^0 G_n^{(p)}) \oplus \HZ_p(G_n^{(p)}; \Fin^1 G_n^{(p)})$.

From this discussion we deduce:

\begin{prop}\label{prop:mattheysMaps}
  The real versions of Matthey's maps from \cref{prop:mattheyKO} are given by
  \begin{equation*}
    \betatC_{p,q} \colon \HZ_p(\Gamma; \Fin^q(\Gamma)) \to \KO_{p+2q}^\Gamma(\Eub{\Gamma}) \tens \C, \quad \phi_\ast [G_n^{(p)}]^q \mapsto \phi_\ast [G_n^{(p)}]_{\KO}^q
  \end{equation*}
\end{prop}

\begin{rem}
  The element $[G_n^{(p)}]_{\KO}^0$ can be rewritten as
  \begin{equation*}
    [G_n^{(p)}]_{\KO}^0 = \sum_{l=0}^{\left\lfloor \frac{n}{2} \right\rfloor}
    ([G^{(p)}]_\KO \times \Re [\omega_n^l] ) \tens 2 \cos(2 \pi l /n ).
  \end{equation*}
    A similar equation involving the sine function holds for $[G_n^{(p)}]_{\KO}^1$.
  Thus it would be possible to restrict to real coefficients.
  However, we shall not use this any further, and continue to keep complex coefficients everywhere.
\end{rem}

\section{Secondary index classes of psc metrics for finite groups}\label{sec:finiteGroups}
The proposition below is essentially due to \citeauthor{BG95Eta}~\cite{BG95Eta}, albeit formulated in a slightly different way.
In the proof, we briefly explain for the convenience of the reader how its statement can be deduced from the result in the literature:
\begin{prop}\label{finiteRationalSurjRho}
  Let $H$ be a finite group and $n \geq 6$.
  Then the $\rho$-invariant $\rho \colon \StolzPos{n-1}{\B H} \to \SG_{n-1}^\R(H)$ is rationally surjective.
\end{prop}
\begin{proof}
  We only need to consider $n$ to be even because the analytic structure group of a finite group rationally vanishes in odd degrees.
  Let $n = 4k + 2q$ with $k \geq 1$ and $q \in \{0,1\}$.
  Each (virtual) unitary representation $\pi$ of $H$ induces a trace functional $\tr_\pi \colon \KU_0(\C H) \to \Z$.
  If $\pi$ is of virtual dimension $0$, then $\tr_\pi$ extends to a functional $\eta_\pi \colon \SG_1^\C(H) \to \R$ on the complex version of the analytic structure group, see~\cite{HR10Eta}.
  By the construction of $\eta_\pi$, the composition $\eta_\pi \circ \rho \colon \StolzPos{4k+2q-1}{\B H} \to \R$ recovers the relative eta invariant used in~\cite{BG95Eta}.

  Since finite groups satisfy the Baum--Connes conjecture, the Higson--Roe sequence rationally becomes a short exact sequence:
  \begin{equation}
    0 \to \Q \to \KU_0(\C H) \tens \Q \to \SG^\C_1(H) \tens \Q \to 0. \label{eq:higsonRoeSeqFin}
  \end{equation}
  Let $\RU_0(H)$ denote the space of virtual unitary representation of dimension $0$.
    The pairing $\RU(H) \tens \Q \times \KU_0(\C H) \tens \Q \to \Q$, $(\pi, x) \mapsto \tr_\pi(x)$ is non-degenerate.
  We conclude from this and \labelcref{eq:higsonRoeSeqFin} that the pairing
  \begin{equation}
    \RU_0(H) \tens \R \times \SG_1^\C(H) \tens \R \to \R, \quad (\pi, x) \mapsto \eta_\pi(x)\label{eq:structPairing}
  \end{equation}
  is also non-degenerate.
The complex analytic structure group admits a decomposition $\SG^\C_{1}(H) \tens \Q \iso (\SG^\R_{1}(H) \oplus \SG^\R_{3}(H)) \tens \Q$ analogous to \cref{prop:KOtoKU}.
Applying this to \labelcref{eq:structPairing} yields a non-degenerate pairing
\begin{equation}
  \RU_0^q(H) \tens \R \times \SG_{2q-1}^\R(H) \tens \R \to \R, \quad (\pi, x) \mapsto \eta_\pi(x)\label{eq:structPairingReal}
\end{equation}
for each $q \in \{0,1\}$.

Finally, let $4k + 2q \geq 6$. Then \cite[Theorem~2.1]{BG95Eta} implies that the composition
\begin{equation*}
  \StolzPos{4k+2q-1}{\B H} \tens \R \xrightarrow{\rho \tens \R} \SG_{2q-1}^\R(H) \tens \R \xrightarrow{\bigoplus_i \eta_{\pi_i}} \R^{\dim \RU^q_0(H)}
\end{equation*}
is surjective, where $(\pi_i)_i$ is a basis of $\RU_0^q(H) \tens \R$.
Since the pairing \labelcref{eq:structPairingReal} is non-degenerate, the latter map in this composition is an isomorphism.
Thus $\rho \tens \R$ must be surjective as well.
\end{proof}

\begin{cor}
  \label{finiteRationalSurj}
    Let $H$ be a finite group and $n \geq 6$.
    Then the relative index map $\alpha \colon \StolzRel{n}{\B H}  \to \KO_{n}(\CstarRed H)$ is rationally surjective.
\end{cor}
\begin{proof}
  Again we only need to consider $n$ to be even and let $n = 4k + 2q \geq 6$.
 For a finite group $H$, the groups $\SpinBordism{l}{\B H}$ and $\KO_l(\B H)$ are torsion for $l \not \equiv 0 \mod 4$.
  Moreover, $\beta \tens \Q \colon \SpinBordism{4k}{\B H} \tens \Q \to \KO_{0}(\B H) \tens \Q$ is surjective because $\KO_{0}(\B H) \tens \Q \iso \KO_0(\pt) \tens \Q$ is generated by the class represented by any product of Kummer surfaces.
  By \cref{finiteRationalSurjRho}, the $\rho$-invariant $\rho \tens \Q \colon   \StolzPos{4k+2q-1}{\B H} \to \SG_{2q-1}^\R(H)$ is also surjective.
  Thus we have a diagram of exact sequences
  \begin{equation*}
    \begin{tikzcd}
      \SpinBordism{4k + 2q}{\B H} \tens \Q \dar["\beta \tens \Q", twoheadrightarrow] \rar &
        \StolzRel{4k+2q}{\B H} \tens \Q \rar \dar["\alpha \tens \Q"] &
        \StolzPos{4k+2q-1}{\B H} \tens \Q  \rar \dar["\rho \tens \Q", twoheadrightarrow] &
        0 \dar[hookrightarrow] \\
        \KO_{2q}(\B H) \tens \Q \rar & \KO_{2q}(\R H) \tens \Q \rar & \SG_{2q-1}^\R(H) \tens \Q \rar & 0,
    \end{tikzcd}
  \end{equation*}
  The \emph{Four Lemma} implies that $\alpha \tens \Q$ must be surjective as well.
\end{proof}

\section{Proof of main results}
Our main result, \Cref{thm:betaPSC}, follows immediately from \cref{prop:mattheysMaps} and the following lemma.

\begin{lem}\label{lem:PSClift}
  For each $n \in \N$, $p \in \{0,1,2\}$, $q \in \{0,1\}$ and $k \geq 1$ with $4k + 2q \geq 6$, there exists $[G_n^{(p)}]_\psc^q \in \StolzRel{p + 2q + 4k}{\B G^{(p)}_n} \tens \C$ with $\alphaAPS([G_n^{(p)}]_\psc^q) = \mu ([G_n^{(p)}]_\KO^q) \in \KO_{2q}(\CstarRed G_n^{(p)}) \tens \C$.
\end{lem}
\begin{proof}
  Let
  \begin{equation*}
    x_{n,l}^q := \begin{cases}
                  \Re[\omega_n^l] & q = 0\\
                  \Im[\omega_n^l] & q = 1
                 \end{cases}
                 \quad \in \KO_{2q}^{\Z/n\Z}(\pt).
  \end{equation*}
  By \cref{finiteRationalSurj}, we can choose an element $y^{q,k}_{n,l} \in \StolzRel{2q + 4k}{\B \Z/n\Z} \tens \Q$ such that $\alphaAPS(y^{q,k}_{n,l}) = \mu(x^q_{n,l}) \in \KO_{2q}(\CstarRed (\Z/n\Z)) \tens \Q$.
  Now let $[G^{(0)}]_\Omega := [\pt] \in \SpinBordism{0}{\pt}$, $[G^{(1)}]_\Omega := [\sphere^1] \in \SpinBordism{1}{\B \Z}$ and $[G^{(2),g}]_\Omega := [\Sigma_g] \in \SpinBordism{2}{\B \Gamma_g} \tens \Q$.
  Note that for the latter we need to choose one from the $2^{2g}$ different spin structures on the oriented surface. 
  However, rationally the element $[\Sigma_g]$ is independent of this choice.
  Taking direct products yields a map $\SpinBordism{l}{X} \tens \StolzRel{m}{Y} \xrightarrow{\times} \StolzRel{l+m}{X \times Y}$.
  Using this, we let
  \begin{equation*}
    [G_n^{(p)}]_\psc^q := \sum_{l=0}^{n-1} \left( [G^{(p)}]_\Omega \times y_{n,l}^{q,k} \right) \tens \omega_n^{-l} \in \StolzRel{p + 2q + 4k}{\B G_n^{(p)}} \tens \C.\qedhere
  \end{equation*}
\end{proof}

\begin{proof}[Proof of \cref{thm:betaPSC}]
  We have the diagram
  \begin{equation*}
    \begin{tikzcd}
    &  \HZ_p(\Gamma; \Fin^q \Gamma ) \rar[hookrightarrow] \dar["\betaPSC_{p,q,k}", dotted] \ar[ddl, bend right, "\betat_{p,q}"]
        & \HZ_p(\Gamma; \Fin \Gamma) \ar[dd, "\betaaC_p"] \\
    &  \StolzRel{p + 2q + 4k}{\B \Gamma} \tens \C \dar["\alpha \tens \C"] & \\
    \KO_{p+2q}^\Gamma(\Eub \Gamma) \tens \C \rar{\mu \tens \C}&  \KO_{p+2q}(\CstarRed \Gamma)  \tens \C \rar["\cxfy \tens \C"]  & \KU_p(\CstarRed \Gamma) \tens \C,
    \end{tikzcd}
  \end{equation*}
  where by construction of $\betat_{p,q}$ the outer paths commute.
  To prove the existence of $\betaPSC_{p,q,k}$, it suffices to show that the image of $\mu \tens \C \circ \betat_{p,q}$ is contained in the image of $\alpha \tens \C$.
  \Cref{prop:mattheysMaps} implies that the image of $\betat_{p,q}$ is generated by elements of the form $\phi_\ast [G_n^{(p)}]_{\KO}^q$, where $[G_n^{(p)}]_{\KO}^q$ is defined in \labelcref{eq:equFundClass} and $\phi \colon G_n^{(p)} \to \Gamma$ is a group homomorphism.
  Thus it sufices to show that the elements $\mu (\phi_\ast [G_n^{(p)}]_{\KO}^q)$ are contained in the image of $\alpha \tens \Q$.
  Indeed, \cref{lem:PSClift} states that $\mu ([G_n^{(p)}]_{\KO}^q)$ admits a lift to $\StolzRel{p+2q+4k}{\B G_n^{(p)}}$.
  By functoriality, we conclude that $\mu( \phi_\ast [G_n^{(p)}]_{\KO}^q )$ admits a lift to $\StolzRel{p+2q+4k}{\B \Gamma}$.
\end{proof}

\begin{proof}[Proof of \cref{cor:RlowerBound}]
  If $\mu \tens \Q$ is injective, then $\betaaC_p = \mu \tens \C \circ \betatC_p$ maps $\HZ_p(\Gamma; \Fin \Gamma)$ injectively into $\KU_p(\CstarRed \Gamma) \tens \C$.
  Thus the diagram in \cref{thm:betaPSC} implies that for fixed $n$ the following map must be injective:
  \begin{equation}
    \sum_{p + 2q + 4k = n} \betaPSC_{p,q,k} \colon \bigoplus_{p+2q \in n + 4\Z} \HZ_p(\Gamma; \Fin^q \Gamma) \to \StolzRel{n}{\B \Gamma} \tens \C. \label{eq:betaPSCInj}
  \end{equation}
  Here $p$, $q$, $k$ range over $\{0,1,2\}$,  $\{0,1\}$, $\Z$, respectively.
  Unpacking this yields the table in the statement of \cref{cor:RlowerBound}.
\end{proof}

\begin{proof}[Proof of \cref{cor:PlowerBound}]
  The image in $\KU_p^\Gamma(\Eub \Gamma) \tens \C$ of the restriction of $\betat_p$ to $\HZ_p(\Gamma; \Fin_0^0 \Gamma \oplus \Fin^1 \Gamma)$ intersects trivially with the image of $\KU_p(\B \Gamma) \tens \C \hookrightarrow \KU_p^\Gamma(\Eub \Gamma) \tens \C$.
  This follows from the decomposition of the handicrafted Chern character based on the Shapiro isomorphism, see~\cite[Theorem~1.4]{matthey:delocChern}.
  Thus the injectivity of \labelcref{eq:betaPSCInj} together with a diagram chase involving \labelcref{eq:mappingPSCtoAnalysis} implies that the following map must be injective as well:
  \begin{equation*}
    \sum_{p + 2q + 4k = n} \partial \circ \betaPSC_{p,q,k} \colon \bigoplus_{p+2q \in n + 4\Z} \HZ_p(\Gamma; \Fin^q_0 \Gamma) \to \StolzPos{n-1}{\B \Gamma} \tens \C.
  \end{equation*}
  Here we use the convention $\Fin^1_0 \Gamma:= \Fin^1 \Gamma$.
\end{proof}

\begin{proof}[Proof of \cref{cor:surj}]
  If the rational homological dimension of $\Gamma$ is at most $2$, then the map
  \begin{equation*}
    \sum_{p + 2q \in n + 4 \Z} \betat_{p,q} \colon \bigoplus_{p + 2q \in n + 4 \Z} \HZ_p(\Gamma; \Fin \Gamma^q) \to \KO_{n}^\Gamma(\Eub \Gamma) \tens \C
  \end{equation*}
  is the inverse to the Chern character.
  In particular, it is surjective.
  If the rational Baum--Connes assembly map $\mu \tens \Q$ is also surjective, then this implies that the following is surjective too:
  \begin{equation*}
    \sum_{p + 2q \in n + 4 \Z} \betaaC_{p,q} = \mu \tens \C \circ \sum_{p + 2q \in n + 4 \Z} \betat_{p,q} \colon \bigoplus_{p + 2q \in n + 4 \Z} \HZ_p(\Gamma; \Fin \Gamma^q) \to \KO_{n}(\CstarRed \Gamma) \tens \C.
  \end{equation*}
  \Cref{thm:betaPSC} implies that for $n \geq 7$, the image of $\betaaC_{p,q}$ is contained in the image of $\alpha \tens \C$, which proves surjectivity of $\alpha \tens \C$ and thus of $\alpha \tens \Q$.

  If $\mu \tens \Q$ is injective, then $\nu \tens \Q$ is injective and by exactness the boundary map $\partial \tens \Q \colon \KO_{n}(\CstarRed \Gamma) \tens \Q \to \SG_{n-1}^\R(\Gamma) \tens \Q$ is surjective.
  Hence the surjectivity statement for $\rho \tens \Q$ if $\mu \tens \Q$ is an isomorphism follows from surjectivity of $\alpha \tens \Q$ and commutativity of the diagram \labelcref{eq:mappingPSCtoAnalysis}.
\end{proof}

\section*{Acknowledgements}
The first-named author thanks the support of PAPIIT-UNAM grants IA
100117 \enquote{Topolog\'ia Equivariante y Teor\'ia de \'indice}, as well as
CONACYT-SEP Foundational Research grant \enquote{Geometr\'ia No
conmutativa, Aplicaci\'on de Baum-Connes y Topolog\'ia Algebraica}.
The second-named author gratefully acknowledges the hospitality of the Centro de Ciencias Matemáticas at UNAM Morelia, where this project was initiated.
Both authors thank Johannes Ebert for valuable discussions.

\printbibliography
\vspace{2ex}
\end{document}